\documentclass[reqno, 12pt]{amsart}
\usepackage{amssymb}
\usepackage{amsfonts}
\usepackage{enumerate}
\usepackage{cite}
\usepackage{amsmath}
\usepackage{amsfonts,amssymb,latexsym}
\usepackage[T1]{fontenc}
\usepackage[cp1250]{inputenc}

\newtheorem{theorem}{Theorem}[section]
\newtheorem{lemma}[theorem]{Lemma}

\theoremstyle{definition}

\newtheorem{example}{Example}

\theoremstyle{remark}
\newtheorem{remark}[theorem]{Remark}

\DeclareMathOperator{\Fix}{Fix}

\begin{document}
\title[A common fixed point theorem]{A common fixed point theorem for \\
a commuting family of weak$^{\ast }$ continuous nonexpansive mappings}
\author[S. Borzdy\'{n}ski]{S\l awomir Borzdy\'{n}ski}
\author[A. Wi\'{s}nicki]{Andrzej Wi\'{s}nicki}

\begin{abstract}
It is shown that if $\mathcal{S}$ is a commuting family of weak$^{\ast }$
continuous nonexpansive mappings acting on a weak$^{\ast }$ compact convex
subset $C$
of the dual Banach space $E,$ then the set of common fixed points of $%
\mathcal{S}$ is a nonempty nonexpansive retract of $C$. This partially
solves a long-standing open problem in metric fixed point theory in the case
of commutative semigroups.
\end{abstract}

\address{S\l awomir Borzdy\'{n}ski, Department of Mathematics, Maria Curie-Sk%
\l odowska University, 20-031 Lublin, Poland}
\email{slawomir.borzdynski@gmail.com}
\address{Andrzej Wi\'{s}nicki, Department of Mathematics, Maria Curie-Sk\l %
odowska University, 20-031 Lublin, Poland}
\email{a.wisnicki@umcs.pl}
\maketitle

\section{Introduction}

A subset $C$ of a Banach space $E$ is said to have the fixed point property
if every nonexpansive mapping $T:C\rightarrow C$ (i.e., $\Vert Tx-Ty\Vert
\leq \Vert x-y\Vert $ for $x,y\in C$) has a fixed point. A general problem,
initiated by the works of F. Browder, D. G\"{o}hde and W. A. Kirk and
studied by numerous authors for over 40 years, is to classify those $E$ and $%
C$ which have the fixed point property. For a fuller discussion on this
topic we refer the reader to \cite{GoKi, KiSi}.

In this paper we concentrate on weak$^{\ast }$ compact convex subsets of the
dual Banach space $E.$ In 1976, L. Karlovitz (see \cite{Ka}) proved that if $%
C$ is a weak$^{\ast }$ compact convex subset of $\ell _{1}$ (as the dual to $%
c_{0}$) then every nonexpansive mapping $T:C\rightarrow C$ has a fixed
point. His result was extended by T.C. Lim \cite{Li} to the case of left
reversible topological semigroups. On the other hand, C. Lennard showed the
example of a weak$^{\ast }$ compact convex subset of $\ell _{1}$ with the
weak$^{\ast }$ topology induced by its predual $c$ and an affine contractive
mapping without fixed points (see \cite[Example 3.2]{Si}). This shows that,
apart from nonexpansiveness, some additional assumptions have to be made to
obtain the fixed points.

Let $S$ be a semitopological semigroup, i.e., a semigroup with a Hausdorff
topology such that for each $t\in S,$ the mappings $s\rightarrow t\cdot s$
and $s\rightarrow s\cdot t$ from $S$ into $S $ are continuous. Consider the
following fixed point property:

\begin{enumerate}
\item[$(F_{\ast })$:] \textit{Whenever $\mathcal{S}=\{T_{s}:s\in S\}$ is a
representation of $S$ as norm-nonexpansive mappings on a non-empty weak$%
^{\ast }$ compact convex set $C$ of a dual Banach space $E$ and the mapping $%
(s,x)\rightarrow T_{s}(x)$ from $S\times C$ to $C$ is jointly continuous,
where $C$ is equipped with the weak$^{\ast }$ topology of $E$, then there is
a common fixed point for $\mathcal{S}$ in $C$.}
\end{enumerate}

It is not difficult to show (see, e.g., \cite[p. 528]{LaTa2}) that property $%
(F_{\ast })$ implies that $S$ is left amenable (in the sense that $LUC(S)$,
the space of bounded complex-valued left uniformly continuous functions on $%
S $, has a left invariant mean). Whether the converse is true is a
long-standing open problem, posed by A. T.-M. Lau in \cite{La2} (see also
\cite[Problem 2]{LaTa2}, \cite[Question 1]{LaZh}).

It is well known that all commutative semigroups are left amenable. The aim
of this paper is to give a partial answer to the above problem by showing
that every commuting family \textit{$\mathcal{S}$ }of weak$^{\ast }$
continuous nonexpansive mappings acting on a weak$^{\ast }$ compact convex
subset $C$ of the dual Banach space $E$ has common fixed points. Moreover,
we prove that the set $\Fix\mathit{\mathcal{S}}$ of fixed points is a
nonexpansive retract of $C.$

Note that the structure of $\Fix\mathit{\mathcal{S}}$ (with $\mathit{%
\mathcal{S}}$ commutative) was examined by R. Bruck (cf. \cite{Br1, Br2})
who proved that if every nonexpansive mapping $T:C\rightarrow C$ has a fixed
point in every nonempty closed convex subset of $C$ which is invariant under
$T$, and $C$ is convex and weakly compact or separable, then $\Fix\mathit{%
\mathcal{S}}$ is a nonexpansive retract of $C.$ We are able to mix the
elements of Bruck's method with some properties of $w^{\ast }$-continuous
and nonexpansive mappings to get the desired result.

\section{Preliminaries}

Let $E$ be the dual of a Banach space $E_{\ast }$. In this paper we focus on
the weak$^{\ast }$ topology -- the smallest one satisfying the condition:
for all $e\in E$, the functional $\hat{e}(x)=x(e)$ is continuous (in the
strong topology). This definition opens up the possibility to consider the
so-called weak$^{\ast }$ properties, for example, $w^{\ast }$-compactness
(compactness in the $w^{\ast }$-topology), $w^{\ast }$-completeness, etc. In
this topology, $E$ becomes a locally convex Hausdorff space. We say that a
dual Banach space $E$ has the $w^{\ast }$-FPP if every nonexpansive
self-mapping defined on a nonempty $w^{\ast }$-compact convex subset of $E$
has a fixed point. It is known that $\ell _{1}=c_{0}^{\ast }$ and some other
Banach lattices have $w^{\ast }$-FPP, however $\ell _{1}=c^{\ast }$ as well
as the duals of $C(\Omega ),$ where $\Omega $ is an infinite compact
Hausdorff topological space, do not possess this property.

A non-void set $D\subset C$ is said to be a nonexpansive retract of $C$ if
there exists a nonexpansive retraction $R:C\rightarrow D$ (i.e., a
nonexpansive mapping $R:C\rightarrow D$ such that $R_{\mid D}=I$). Since we
deal a lot with $w^{\ast }$-continuous nonexpansive mappings, we abbreviate
them to $w^{\ast }$-CN.

We conclude with recalling the following consequence of the Ishikawa theorem
(see \cite{Is}): if $C$ is a bounded convex subset of a Banach space $X$, $%
\gamma \in (0,1),$ and $T:C\rightarrow C$ is nonexpansive, then the mapping $%
T_{\gamma }=(1-\gamma )I+\gamma T$ is asymptotically regular, i.e.,
$\lim_{n\rightarrow \infty }\left\Vert T_\gamma^{n+1}x-T_\gamma^{n}x\right\Vert =0$ for
every $x\in C.$ We use this theorem in Lemma \ref{finiteCommuting}.

\section{Fixed-point theorems}

We begin with a structural result concerning a single $w^{\ast }$-continuous
nonexpansive mapping $T:C\rightarrow C.$

\begin{theorem}
\label{thm1} Let $C$ be a nonempty weak$^{\ast }$ compact convex subset of
the dual Banach space. Then for any $w^{\ast }$-CN self-mapping $T$ on $C$,
the set $\Fix T$ of fixed points of $T$ is a (nonempty) nonexpansive retract
of $C$.
\end{theorem}

The proof will follow by constructing gradually (and establishing properties
of) three functions, each one defined in the means of the earlier, and the
last one being the retraction from $C$ to $\Fix T$.

\begin{proof}
Notice first that $C$ is complete in the strong topology. Now, for $x\in C$
and a postive integer $n$, consider a mapping $T_{x}:C\rightarrow C$ defined
by

\begin{equation*}
T_{x}z=\frac{1}{n}x+\left( 1-\frac{1}{n}\right) Tz,\;z\in C.
\end{equation*}%
It is not difficult to see that $T_{x}$ is a contraction:
\begin{equation*}
\left\Vert T_{x}y-T_{x}z\right\Vert \leq (1-\frac{1}{n})\left\Vert
y-z\right\Vert .
\end{equation*}%
Hence and from completeness of $C$, it follows from the Banach Contraction
Principle that there exists exactly one point $F_{n}x\in C$ such that $%
T_{x}F_{n}x=F_{n}x$. This defines a mapping $F_{n}:C\rightarrow C$ by%
\begin{equation}
F_{n}x=\frac{1}{n}x+\left( 1-\frac{1}{n}\right) TF_{n}x  \label{mapping}
\end{equation}%
for $x\in C$. Thus

\begin{equation*}
\left\Vert TF_{n}x-F_{n}x\right\Vert =\frac{1}{n}\left\Vert
TF_{n}x-x\right\Vert \leq \frac{1}{n}\mathrm{diam\,}C
\end{equation*}%
and consequently,

\begin{equation*}
\lim_{n}\left\Vert TF_{n}x-F_{n}x\right\Vert =0
\end{equation*}%
since $C$ is bounded in norm as a weak$^{\ast }$ compact subset of a Banach
space.

Notice that for $x\in \Fix T$ we have
\begin{equation*}
T_{x}x=x
\end{equation*}%
and consequently $F_{n}x=x.$

Furthermore, $F_{n}x$ is nonexpansive. Indeed,%
\begin{equation}
F_{n}x-F_{n}y=T_{x}F_{n}x-T_{y}F_{n}y=\frac{1}{n}(x-y)+\left( 1-\frac{1}{n}%
\right) (Tx-Ty)  \label{nonexpansive}
\end{equation}%
which, by putting it into norm and using the triangle inequality and
nonexpansiveness of $T$, gives us a desired statement.

Notice that we can view $C^{C}$ as the product space of copies of $C$, where
each copy is endowed with the w$^{\ast }$-topology. Then, according to
Tychonoff's theorem, $C^{C}$ is compact in the product topology generated in
this way (\textquotedblleft w$^{\ast }$-product topology\textquotedblright
). It follows that a sequence $(F_{n})_{n\in \mathbb{N}}$ of elements from $%
C^{C}$ has a convergent subnet $(F_{n_{\alpha }})_{\alpha \in \Lambda }$ and
we can define
\begin{equation*}
R=w^{\ast }\text{-}\lim_{\alpha }F_{n_{\alpha }},
\end{equation*}%
where the above limit should be understood as taken in the aforementioned $%
w^{\ast }$-product topology. Now we can treat the application of $R$ to some
$x\in C$ as the projection of the mapping onto the $x$-th coordinate and
since such projections are continuous in the product topology, we obtain%
\begin{equation*}
Rx=w^{\ast }\text{-}\lim_{\alpha }F_{n_{\alpha }}x,
\end{equation*}%
where this limit is an ordinary $w^{\ast }$-limit. With this approach, we
are able to construct one subnet which guarantees convergence for all $x\in
C $.

Notice that%
\begin{equation*}
TRx=w^{\ast }\text{-}\lim_{\alpha }TF_{n_{\alpha }}x
\end{equation*}%
since $T$ is weak$^{\ast }$ continuous. Now, it follows from the weak$^{\ast
}$ lower semicontinuity of the norm that for any $x\in C,$

\begin{equation*}
\left\Vert TRx-Rx\right\Vert =\left\Vert w^{\ast }\text{-}\lim_{\alpha
}(TF_{n_{\alpha }}x-F_{n_{\alpha }}x)\right\Vert \leq \liminf_{\alpha
}\left\Vert TF_{n_{\alpha }}x-F_{n_{\alpha }}x\right\Vert =0
\end{equation*}%
and hence
\begin{equation*}
TRx=Rx
\end{equation*}%
which means that $Rx\in \Fix T$. Furthermore, $Rx=x$ if $x\in \Fix T$.

We can now use (\ref{nonexpansive}) and the weak$^{\ast }$ lower
semicontinuity of the norm to prove that $R$ is nonexpansive:
\begin{align*}
\left\Vert Rx-Ry\right\Vert & =\left\Vert \text{w}^{\ast }\text{-}%
\lim_{\alpha }(F_{n_{\alpha }}x-F_{n_{\alpha }}y)\right\Vert  \\
& \leq \liminf_{\alpha }\left\Vert \frac{1}{n_{\alpha }}(x-y)+(1-\frac{1}
{n_{\alpha }})(Tx-Ty)\right\Vert \leq \limsup_{\alpha }\frac{1}{n_{\alpha }}
\left\Vert x-y\right\Vert  \\
& +\limsup_{\alpha }(1-\frac{1}{n_{\alpha }})\left\Vert Tx-Ty\right\Vert
=\left\Vert Tx-Ty\right\Vert \leq \left\Vert x-y\right\Vert .
\end{align*}%
Thus we conclude that $\Fix T$ is indeed a nonexpansive retract of $C.$
\end{proof}

\begin{remark}
The $w^{\ast }$-continuity of $T$ cannot be omitted in the assumptions of
Theorem \ref{thm1}. Indeed, otherwise we would conclude that any dual Banach
space has $w^{\ast }$-FPP. But it is known (see, e.g., \cite[Example 3.2]{Si}%
) that $\ell _{1}$ (as the dual to the Banach space $c$) fails the $w^{\ast
} $-FPP, a contradiction.
\end{remark}

The following example shows that we would not be able to relax the
assumption of the nonexpansiveness of $T$ to continuity, either, even if we
only postulated the existence of a (continuous) retraction.

\begin{example}
Let $\ell _{1}=c_{0}^{\ast }$ and define
\begin{equation*}
T(x_{1},x_{2},x_{3},...)=((x_{1})^{2},0,x_{2},x_{3},...)
\end{equation*}
on the unit ball $B_{\ell _{1}}.$ Notice that $T:B_{\ell _{1}}\rightarrow
B_{\ell _{1}}$ is $w^{\ast }$-continuous and $\Fix T=\{{(\pm 1,0,...)\}.}$
But a non-connected set cannot be a retract of the ball.
\end{example}

Our next objective is to generalize Theorem \ref{thm1} to a commuting family
of $w^{\ast }$-continuous nonexpansive mappings. If $\mathcal{S}%
=\{T_{s}:s\in S\}$ is a family of mappings, we denote by%
\begin{equation*}
\Fix\mathcal{S}=\bigcap_{s\in S}\Fix T_{s}
\end{equation*}%
the set of common fixed points of $\mathcal{S}$.

We first prove a lemma which resembles \cite[Lemma 6]{Br1}.

\begin{lemma}
\label{lem1} Let $\mathcal{S}$ be a family of commuting self-mappings acting
on a set $C$ and suppose that there exists a retraction $R$ of $C$ onto
$\Fix\,\mathcal{S}$. If $\widetilde{T}$ commutes with every element of the
family $\mathcal{S}$, then
\begin{equation*}
\Fix\mathcal{S}\cap \Fix\widetilde{T}=\Fix(\widetilde{T}R).
\end{equation*}
\end{lemma}

\begin{proof}
The inclusion from left to right follows from the simple observation that if
$x\in \Fix\mathcal{S}\cap \Fix\widetilde{T},$ then $Rx=x$ and $\widetilde{T}%
x=x.$

For the other direction, assume $x\in \Fix(\widetilde{T}R)$ which means $%
\widetilde{T}Rx=x$. Then, for every $T\in \mathcal{S}$, it follows from the
commutativity and the fact that $Rx\in \Fix T$ that
\begin{equation*}
T\widetilde{T}Rx=\widetilde{T}(TRx)=\widetilde{T}Rx.
\end{equation*}

Therefore $\widetilde{T}Rx\in \Fix T$ for every $T\in \mathcal{S}$ and
consequently%
\begin{equation*}
x=\widetilde{T}Rx\in \Fix\mathcal{S}.
\end{equation*}%
Since $R$ is a retraction onto $\Fix\mathcal{S}$, we have $Rx=x$ and hence $%
\widetilde{T}x=x.$ It follows that $x\in \Fix\mathcal{S}\cap \Fix\widetilde{T%
}$ which proves the inclusion and the whole lemma.
\end{proof}

\begin{lemma}
\label{finiteCommuting} Suppose that $C$ is as in Theorem \ref{thm1} and $%
\mathcal{S}_{n}=\{T_{1},...,T_{n}\}$ is a finite commuting family of $%
w^{\ast }$-CN self-mappings on $C$. Then $\Fix\mathcal{S}_{n}$ is a
nonexpansive retract of $C.$
\end{lemma}

\begin{proof}
We will show by induction on $n$ that there exists a nonexpansive retraction
$R_{n}$ from $C$ onto $\Fix\mathcal{S}_{n}$. The the base case $n=1$ follows
directly from Theorem \ref{thm1} since $\Fix\mathcal{S}_{1}=\Fix T_{1}.$

Now assume that that there exists a nonexpansive retraction $R_{n}$ of $C$
onto $\Fix\mathcal{S}_{n}$. We need to show the existence of a nonexpansive
retraction $R_{n+1}$ of $C$ onto $\Fix\mathcal{S}_{n+1}$.

Let%
\begin{equation*}
\widetilde{R}_{n}x=\frac{1}{2}x+\frac{1}{2}T_{n+1}R_{n}x,\ x\in C,
\end{equation*}%
and consider a sequence $(\widetilde{R}_{n}^{k})_{k\in \mathbb{N}}$ of
successive iterations of $\widetilde{R}_{n}.$ As in the proof of Theorem \ref%
{thm1}, we can view $C^{C}$ as the product space, compact with respect to
the w$^{\ast }$-topology on $C.$ Hence the sequence $(\widetilde{R}%
_{n}^{k})_{k\in \mathbb{N}}$ has a convergent subnet $(\widetilde{R}%
_{n}^{k_{\alpha }})_{\alpha \in \Lambda }$ and we can define
\begin{equation*}
R_{n+1}x=w^{\ast }\text{-}\lim_{\alpha }\widetilde{R}_{n}^{k_{\alpha }}x
\end{equation*}%
for every $x\in C.$

Since $T_{n+1}R_{n}$ is nonexpansive as a composition of such mappings, it
is easy to see that also $\widetilde{R}_{n}$ is nonexpansive. The
nonexpansiveness of $R_{n+1}$ now follows from the weak$^{\ast }$ lower
semicontinuity of the norm. It is also easy to see that $\Fix %
T_{n+1}R_{n}\subset \Fix R_{n+1}$ and, by using Lemma \ref{lem1}, we
conclude that%
\begin{equation*}
\Fix\mathcal{S}_{n+1}\subset \Fix R_{n+1}.
\end{equation*}%
But this still does not prove that $R_{n+1}$ is a mapping we are looking
for, nor that $\Fix\mathcal{S}_{n+1}$ is nonempty. To complete the proof, we
must show that $R_{n+1}$ is a mapping onto $\Fix\mathcal{S}_{n+1}.$ The rest
of the proof is about showing this fact.

Since $C$ is convex closed and bounded, and $\widetilde{R}_{n}$ is the
convex combination of a nonexpansive mapping and the identity, it follows
from the Ishikawa theorem \cite{Is} that $\widetilde{R}_{n}$ is
asymptotically regular, i.e.,%
\begin{equation*}
\lim_{k\rightarrow \infty }\left\Vert \widetilde{R}_{n}^{k+1}x-\widetilde{R}%
_{n}^{k}x\right\Vert =0
\end{equation*}%
for every $x\in C$.

Now, fix $x$ and notice that $(\widetilde{R}_{n}^{k_{\alpha }}x)_{\alpha \in
\Lambda }$ is an approximate fixed point net for the mapping $T_{n+1}R_{n}$.
To see this, use the equation%
\begin{equation*}
\widetilde{R}_{n}^{k_{\alpha }+1}x=\frac{1}{2}\left( \widetilde{R}%
_{n}^{k_{\alpha }}x-T_{n+1}R_{n}\widetilde{R}_{n}^{k_{\alpha }}x\right)
+T_{n+1}R_{n}\widetilde{R}_{n}^{k_{\alpha }}x
\end{equation*}%
and the asymptotical regularity in the following calculations:%
\begin{align*}
\limsup_{\alpha }\left\Vert T_{n+1}R_{n}\widetilde{R}_{n}^{k_{\alpha }}x-%
\widetilde{R}_{n}^{k_{\alpha }}x\right\Vert & \leq \limsup_{\alpha
}\left\Vert T_{n+1}R_{n}\widetilde{R}_{n}^{k_{\alpha }}x-\widetilde{R}%
_{n}^{k_{\alpha }+1}x\right\Vert \\
+\lim_{\alpha }\left\Vert \widetilde{R}_{n}^{k_{\alpha }+1}x-\widetilde{R}%
_{n}^{k_{\alpha }}x\right\Vert & =\limsup_{\alpha }\left\Vert T_{n+1}R_{n}%
\widetilde{R}_{n}^{k_{\alpha }}x-\widetilde{R}_{n}^{k_{\alpha
}+1}x\right\Vert \\
& =\frac{1}{2}\limsup_{\alpha }\left\Vert T_{n+1}R_{n}\widetilde{R}%
_{n}^{k_{\alpha }}x-\widetilde{R}_{n}^{k_{\alpha }}x\right\Vert .
\end{align*}%
Thus we conclude that%
\begin{equation}
\lim_{\alpha }\left\Vert T_{n+1}R_{n}\widetilde{R}_{n}^{k_{\alpha }}x-%
\widetilde{R}_{n}^{k_{\alpha }}x\right\Vert =0,  \label{map1Lim}
\end{equation}%
as desired.

Now, for brevity, denote $r_{\alpha }=\widetilde{R}_{n}^{k_{\alpha }}x$ and
notice that for every $m\leq n$
\begin{equation*}
T_{m}T_{n+1}R_{n}r_{\alpha }=T_{n+1}T_{m}R_{n}r_{\alpha
}=T_{n+1}R_{n}r_{\alpha }.
\end{equation*}%
That is, $T_{n+1}R_{n}r_{\alpha }\in \Fix T_{m}$ which is equivalent to the
statement that $T_{n+1}R_{n}r_{\alpha }$ belongs to $\Fix\mathcal{S}_{n}$.
It follows that
\begin{equation*}
T_{n+1}R_{n}r_{\alpha }=R_{n}T_{n+1}R_{n}r_{\alpha }.
\end{equation*}%
and using the equation (\ref{map1Lim}), we obtain
\begin{align}
& \limsup_{\alpha }\left\Vert R_{n}r_{\alpha }-r_{\alpha }\right\Vert \leq
\limsup_{\alpha }\left\Vert R_{n}r_{\alpha }-T_{n+1}R_{n}r_{\alpha
}\right\Vert +\lim_{\alpha }\left\Vert T_{n+1}R_{n}r_{\alpha }-r_{\alpha
}\right\Vert  \label{map2Lim} \\
& =\limsup_{\alpha }\left\Vert R_{n}r_{\alpha }-R_{n}T_{n+1}R_{n}r_{\alpha
}\right\Vert \leq \lim_{\alpha }\left\Vert r_{\alpha }-T_{n+1}R_{n}r_{\alpha
}\right\Vert =0.  \notag
\end{align}%
In the same manner we can see that for every $m\leq n$,%
\begin{align*}
\limsup_{\alpha }\left\Vert T_{m}r_{\alpha }-r_{\alpha }\right\Vert & \leq
\limsup_{\alpha }\left\Vert T_{m}r_{\alpha }-T_{m}R_{n}r_{\alpha
}\right\Vert +\limsup_{\alpha }\left\Vert T_{m}R_{n}r_{\alpha }-r_{\alpha
}\right\Vert \\
& \leq \lim_{\alpha }\left\Vert r_{\alpha }-R_{n}r_{\alpha }\right\Vert
+\lim_{\alpha }\left\Vert R_{n}r_{\alpha }-r_{\alpha }\right\Vert =0.
\end{align*}%
Since $T_{m}$ is $w^{\ast }$-continuous, this easily yields%
\begin{equation*}
T_{m}R_{n+1}x=R_{n+1}x
\end{equation*}%
and, consequently,
\begin{equation}
R_{n+1}x\in \Fix\mathcal{S}_{n}.  \label{map3Eq}
\end{equation}%
Finally, by using (\ref{map1Lim}) and (\ref{map2Lim}), we get%
\begin{align*}
& \limsup_{\alpha }\left\Vert T_{n+1}r_{\alpha }-r_{\alpha }\right\Vert \leq
\limsup_{\alpha }\left\Vert T_{n+1}r_{\alpha }-T_{n+1}R_{n}r_{\alpha
}\right\Vert +\lim_{\alpha }\left\Vert T_{n+1}R_{n}r_{\alpha }-r_{\alpha
}\right\Vert \\
& +\lim_{\alpha }\left\Vert T_{n+1}R_{n}r_{\alpha }-r_{\alpha }\right\Vert
\leq \lim_{\alpha }\left\Vert r_{\alpha }-R_{n}r_{\alpha }\right\Vert =0.
\end{align*}%
Then, from the $w^{\ast }$-continuity of $T_{n+1}$,
\begin{equation*}
T_{n+1}R_{n+1}x=R_{n+1}x
\end{equation*}%
which combined with (\ref{map3Eq}), gives
\begin{equation*}
R_{n+1}x\in \Fix\mathcal{S}_{n+1}.
\end{equation*}%
That is, $\Fix\mathcal{S}_{n+1}$ is nonempty and $R_{n+1}$ acts onto it,
which completes the proof.
\end{proof}

We are now in a position to prove our main theorem.

\begin{theorem}
\label{thm3} Suppose that $C$ is as in Theorem \ref{thm1} and $\mathcal{S}$
is an arbitrary family of commuting $w^{\ast }$-CN self-mappings on $C$.
Then $\Fix\mathcal{S}$ is a nonexpansive retract of $C.$
\end{theorem}

\begin{proof}
If $\mathcal{S}$ is finite, we can use lemma \ref{finiteCommuting}. So
assume that $\mathcal{S}$ is infinite. First notice that
\begin{equation*}
\Fix T=(T-I)^{-1}\{0\}
\end{equation*}%
is closed in the $w^{\ast }$-topology for every $T\in \mathcal{S}$. Let%
\begin{equation*}
\Lambda =\{\alpha \subset \mathcal{S}:\#\alpha <\infty \}
\end{equation*}%
be a directed set with the inclusion relation $\leq $. Denote by $R_{\alpha
} $ the nonexpansive retraction from $C$ to $\Fix_{\alpha }=\bigcap_{T\in
\alpha }\Fix T$ (a more convenient way of writing $\Fix\alpha $) which
existence is guaranteed by Lemma \ref{finiteCommuting}. Then we have a net $%
(R_{\alpha })_{\alpha \in \Lambda }$, and we can select a subnet $(R_{\alpha
_{\gamma }})_{\gamma \in \Gamma }$, $w^{\ast }$-convergent for any $x\in C$.
Define%
\begin{equation*}
Rx=w^{\ast }\text{-}\lim_{\gamma }R_{\alpha _{\gamma }}x.
\end{equation*}%
For a fixed $T\in \,\mathcal{S}$, take $\gamma _{0}$ such that $\alpha
_{\gamma }\geq \{T\}$ for every $\gamma \geq \gamma _{0}$. It exists,
straightforward from the subnet definition. Then%
\begin{equation*}
\forall _{\gamma \geq \gamma _{0}}\,R_{\alpha _{\gamma }}x\in \Fix_{\alpha
_{\gamma }}\subset \Fix_{\alpha _{\gamma _{0}}}\subset \Fix T
\end{equation*}%
and hence $R_{\alpha _{\gamma }}x$ lies eventually in the $w^{\ast }$-closed
set $\Fix T$. Therefore, $Rx\in \Fix T$ for every $T\in \mathcal{S}$ which
implies $Rx\in \Fix\mathcal{S}$. It is easy to see that $R$ is nonexpansive.
Also, for every $\alpha $,%
\begin{equation*}
x\in \Fix\mathcal{S}\implies x\in \Fix_{\alpha }\implies R_{\alpha }x=x,
\end{equation*}%
from which follows%
\begin{equation}
Rx=x,\,x\in \Fix\mathcal{S}.
\end{equation}%
Thus $R$ is a nonexpansive retraction from $C$ onto $\Fix\mathcal{S}$.
\end{proof}

\begin{remark}
In particular, the set $\Fix\mathcal{S}$ is non-empty. Thus Theorem \ref%
{thm3} answers affirmatively \cite[Question 1]{LaZh} in the case of
commutative semigroups.
\end{remark}

\end{document}